\documentclass[11pt,a4paper]{amsart}

\usepackage[T1]{fontenc}
\usepackage[utf8]{inputenc}

\usepackage[english]{babel}
\usepackage{amsmath,amsfonts,amssymb,amsthm}

\usepackage[pdftex]{graphicx}
\usepackage{lmodern}

\usepackage[bottom=3cm, top=3cm, left=3.5cm, right=3.5cm]{geometry}

\usepackage[foot]{amsaddr}

\usepackage{mathtools}
\mathtoolsset{showonlyrefs,showmanualtags}	

\DeclareMathAlphabet{\mathbbold}{U}{bbold}{m}{n}	

\theoremstyle{plain}
\newtheorem{theorem}{Theorem}
\newtheorem*{theorem*}{Theorem}
\newtheorem{proposition}[theorem]{Proposition}
\newtheorem{cor}[theorem]{Corollary}

\newtheorem{lemma}[theorem]{Lemma}
\newtheorem*{lemma*}{Lemma}
\newtheorem*{prob}{Open problem}

\theoremstyle{definition}
\newtheorem{definition}[theorem]{Definition}
\theoremstyle{remark}

\newtheorem{remark}[theorem]{Remark}
\newtheorem{rmk}[theorem]{Remark}

\usepackage[breaklinks=true]{hyperref}
\usepackage{xcolor}
\hypersetup{
    colorlinks,
    linkcolor={red!50!black},
    citecolor={blue!50!black},
    urlcolor={blue!80!black}
}

\usepackage[alphabetic, abbrev]{amsrefs}

\newcommand{\N}{\mathbb N}

\newcommand{\R}{\mathbb R}

\newcommand{\distr}{\mathcal{D}}

\DeclareMathOperator{\Span}{span}

\allowdisplaybreaks

\title{Branching geodesics in sub-Riemannian geometry}
\author[Thomas Mietton]{Thomas Mietton$^\dagger$}
\author[Luca Rizzi]{Luca Rizzi$^\dagger$}
\address{$^\dagger$Univ. Grenoble Alpes, CNRS, Institut Fourier, 38000 Grenoble, France}

\usepackage{todonotes}

\date{\today}


\begin{document}

\begin{abstract}
In this note, we show that sub-Riemannian manifolds can contain branching normal minimizing geodesics. This phenomenon occurs if and only if a normal geodesic has a discontinuity in its rank at a non-zero time, which in particular for a strictly normal geodesic means that it contains a non-trivial abnormal subsegment. The simplest example is obtained by gluing the three-dimensional Martinet flat structure with the Heisenberg group in a suitable way. 
We then use this example to construct more general types of branching.
\end{abstract}
	
\maketitle

\setcounter{tocdepth}{1}
\tableofcontents

\section{Introduction}

A metric space is branching if there exist minimal geodesics starting from the same point which follow the same path for some initial time interval, and then become disjoint. Common examples are found in Finsler geometry (e.g.\ $\R^2$ with the sup norm), or on graphs. On the other hand, it is well-known that Riemannian manifolds and Alexandrov spaces with curvature bounded from below are non-branching.

We are interested here in sub-Rieman\-nian spaces, a large class of metric structures generalizing Riemannian geometry where a metric is defined only on a subset of tangent directions (cf.\ Section \ref{s:sr} for precise definitions). Several questions concerning geodesics, which are trivial in Riemannian geometry, become hard open problems in the sub-Riemannian setting. For example it has been only recently proven in \cite{corner} that sub-Riemannian geodesics cannot have corners, but it is not yet known whether geodesics are $C^1$, see for example \cite{R-Bourbaki}.

To provide further motivation for our contribution, let us mention that there is an on-going effort in trying to define a suitable concept of lower curvature bound for sub-Riemannian spaces, in particular in relation with the synthetic approach à la Lott-Sturm-Villani (cf.\ for example \cites{BR-interpolation,BKS1,BKS2,BG-JEMS,Milman-QCD} and \cite[p.\ 1127–43]{V-Bourbaki}). Since the existence of branching geodesics causes difficulties in the study of optimal transport and of spaces satisfying synthetic curvature bounds, it is important for further progress in the theory to understand whether sub-Riemannian structures can exhibit such a phenomenon.

In this paper we show that sub-Riemannian (normal) geodesics can branch, adding this phenomenon to the list of remarkable features of sub-Riemannian geometry. Even though normal geodesics are obtained through the action of a Hamiltonian flow, and are therefore smooth, they are not uniquely characterized by their jet at some point. To our best knowledge, it is the first time that this fact is observed.

\subsection{Branching and magnetic fields}\label{s:magnetic}

We describe succinctly the branching phenomenon through an example, of the same basic type as Montgomery's original construction of abnormal minimizers. We will exploit the connection between the equations of motion of a particle in a magnetic field and sub-Riemannian geodesics, pointed out in \cite{Mont-iso,Mont-abnormal}. 

Take the distribution on $\R^3$ defined by the kernel of a one-form $\omega = dz-A(x,y)dy$, and consider the sub-Riemannian metric given by the restriction of $dx^2+dy^2$. Let $B = \partial_x A$ be the magnetic field associated with $\omega$, that is $\omega \wedge d\omega = -B(x,y)dx\wedge dy\wedge dz$. Notice that, without changing $B$, we can alter $A$ in such a way that $A(0,y)=0$ (gauge freedom), and thus the straight line $\gamma_0(t) = (0,t,0)$ for $t\in \R$ is horizontal.

It is well-known that abnormal paths are precisely the horizontal curves contained in the zero-locus of $B$, while normal geodesics are those whose projection $(x(t),y(t))$ on the $xy$-plane satisfies
\begin{equation}\label{eq:magnetic}
\kappa(t) = \lambda B(x(t),y(t)),
\end{equation}
for some $\lambda \in \R$, and where $\kappa(t)$ is the curvature of $(x(t),y(t))$, which we assume to be parametrized with constant speed. In particular, the ODE corresponding to \eqref{eq:magnetic} describes the motion of a particle with charge $\lambda$ under the action of the magnetic field $B(x,y)$ normal to the plane. 

Choose a smooth potential $A(x,y)$ such that $B(x,y)=x$ for $y<0$ and $B(0,y)>0$ when $y>0$.  The zero-locus of $B$ coincides in this case with $x=0$ when $y<0$. In particular, $\gamma_0(t) = (0,t,0)$, for $t<0$, is an abnormal geodesic. Since $B(0,y)=0$ for $y<0$, the curve $\gamma_0(t)$ for $t<0$ satisfies also \eqref{eq:magnetic} for any $\lambda \in \R$, so it is also normal. We can now extend such a curve to a normal geodesic $\gamma_\lambda(t)$ for $t \in \R$, by solving \eqref{eq:magnetic} for different values of $\lambda \in \R$. Of course, $\gamma_0$ corresponds to the straight line but, from the fact that $B(0,y)>0$ for $y>0$, the curve $\gamma_\lambda$ must have non-vanishing curvature for small non-zero $\lambda$, hence a branching phenomenon occurs at the origin.
Moreover, from what we will show in Proposition \ref{p:spray}, the projection on the $xy$-plane of the trajectories $\gamma_{\lambda}$ contains an open neighborhood of the positive $y$-axis and those trajectories are all distinct for $\lambda$ sufficiently small. From the physical viewpoint, this phenomenon corresponds to particles having different charges, which ``spray out'' following different trajectories under the influence of the magnetic field.

In Section \ref{s:example1}, we show an explicit construction of such $A(x,y)$, obtained by gluing a flat Martinet structure with the standard Heisenberg one.

\subsection{Strictly abnormal branching}
In this note we only consider the branching of normal geodesics, and we do not cover the possible branching of strictly abnormal ones (cf.\ Section \ref{s:sr}). It is easy to produce sub-Riemannian structures with branching abnormal paths. For example, consider a degenerate Martinet-type structure in a three-dimensional space, whose Martinet surface itself branches. Such a structure cannot verify the usual non-degeneracy condition, cf.\ \cite[Sec.\ 3.2]{montgomery02}. The Liu-Sussmann local minimality result for abnormal paths does not apply \cite{LS-memoir}, and we are not able to prove that these paths are geodesic (i.e.\ length-minimizing curves).
\begin{prob}
Find an example of branching strictly abnormal geodesics.
\end{prob}

\subsection*{Structure of the paper}
To make the paper self-contained, in Section \ref{s:sr} we recall some basic facts in sub-Riemannian geometry, following \cite{abb19}. Sub-Riemannian branching is then discussed in Section \ref{s:branch}, and put in relation with the well-known normal/abnormal duality of sub-Riemannian geodesics. The explicit construction of an example, as described in the abstract, is done in Section \ref{s:example1}. We conclude by building the most general case of multiply-branching normal geodesics, in Section \ref{s:multiple}.

\section{Sub-Riemannian geometry}\label{s:sr}

A sub-Rieman\-nian structure on a smooth $n$-dimensional manifold $M$, where $n\geq 2$, is defined by a set of $m$ global smooth vector fields $X_{1},\ldots,X_{m}$, called a \emph{generating frame}. The \emph{distribution} is the possibly rank-varying family of subspaces of the tangent spaces spanned by the vector fields at each point
\begin{equation}
\distr_{x}=\Span\{X_{1}(x),\ldots,X_{m}(x)\}\subseteq T_{x}M,\qquad \forall\, x\in M.
\end{equation}
The generating frame induces an inner product $g_{x}$ on $\distr_{x}$  such that
\begin{equation}
g_{x}(v,v):=\min\left\{\sum_{i=1}^{m}u_{i}^{2}\mid \sum_{i=1}^{m}u_{i}X_{i}(x)=v\right\}, \qquad \forall v\in\distr_x.
\end{equation}
We assume that the structure is \emph{bracket-generating}, i.e., the tangent space $T_{x}M$ is spanned by the vector fields $X_{1},\ldots,X_{m}$ and their iterated Lie brackets at $x$.

A \emph{horizontal curve} $\gamma : [0,1] \to M$ is an absolutely continuous path such that there exists a \emph{control} $u\in L^{2}([0,1],\R^{m})$ satisfying
\begin{equation}\label{eq:horizontal}
\dot\gamma(t) =  \sum_{i=1}^m u_i(t) X_i(\gamma(t)), \qquad \mathrm{a.e.}\, t \in [0,1].
\end{equation}
This implies that $\dot\gamma(t) \in \distr_{\gamma(t)}$ for almost every $t$. Notice that the control $u$ in \eqref{eq:horizontal} is not unique, but one can always find a unique \emph{minimal control}, i.e.\ the one such that $g_x(\dot\gamma(t),\dot\gamma(t)) = |u(t)|^2$ for a.e. $t\in [0,1]$. We define the \emph{length} of $\gamma$ as
\begin{equation}
\ell(\gamma) = \int_0^1 \sqrt{g(\dot\gamma(t),\dot\gamma(t))}dt.
\end{equation}
The \emph{sub-Rieman\-nian} (or  \emph{Carnot-Carathéodory})  \emph{distance} is defined by:
\begin{equation}
d_{SR}(x,y) = \inf\{\ell(\gamma)\mid \gamma(0) = x,\, \gamma(1) = y,\, \gamma \text{ horizontal} \}.
\end{equation}




By the Chow-Rashevskii theorem, under the bracket-generating assumption, between any two points $x,y\in M$, there exists a horizontal path, and therefore the sub-Riemannian distance is well-defined. Furthermore, one can prove that it is continuous and its induced topology is the same as the manifold topology. We remark that this definition of sub-Riemannian metric, based on the concept of global generating frame, includes the classical constant-rank case, see \cite[Section 3.1.4]{abb19}.

In place of the length, it is convenient to consider the \emph{energy} 
\begin{equation}
J(\gamma) = \frac{1}{2}\int_0^1 g(\dot\gamma(t),\dot\gamma(t)) dt.
\end{equation}
Horizontal trajectories minimizing the energy with fixed endpoints are exactly paths that minimize the length, parametrized with constant speed. We call a \emph{minimizing geodesic} between two points $x$ and $y$ in $M$ a horizontal path $\gamma :[0,1]\to M$, with $\gamma(0)=x$ and $\gamma(1)=y$, that minimizes the energy among all horizontal paths sharing the same extremities. The term \emph{geodesic}, instead, denotes the more general class of horizontal paths that are minimizing geodesics locally around each of its points.


\subsection{Characterization of geodesics}

Let $x\in M$. By Cauchy-Lipschitz, there exists an open set $\mathcal{U} \subseteq L^2([0,1],\mathbb{R}^m)$ such that, for all $u \in \mathcal{U}$, the Cauchy problem:
\begin{equation}
\dot{\gamma}_u(t) = \sum_{i=1}^m u_i(t) X_i(\gamma_u(t)), \qquad \gamma_u(0)=x
\end{equation}
admits a solution $\gamma_u$ defined on $[0,1]$. The \emph{end-point map} $E_x : \mathcal{U} \mapsto M$ is then $E_x(u) = \gamma_u(1)$. The end-point map is weakly continuous and differentiable. The problem of finding a minimizing geodesic between two points $x$ and $y$ is then the problem of minimizing the functional $J$ (seen as a smooth functional defined on $\mathcal{U}$) under the constraint $E_x(u)=y$. By the Lagrange's multipliers rule, if $\gamma$ is a minimizing geodesic between $x$ and $y$, and $u$ is its minimal control, then there exists $\lambda_1 \in T_y^*M$ and $\nu \in \{0,1\}$, with $(\lambda_1,\nu) \neq 0$, such that 
\begin{equation}\label{multLagranage}
	\lambda_1 \circ D_u E_x = \nu \langle u,\cdot \rangle,
\end{equation}
where $\circ$ denotes the composition of linear maps and $D$ the (Fréchet) differential. Any path whose minimal control verifies \eqref{multLagranage} with $\nu=0$ is called \emph{abnormal}, or \emph{singular}. A path verifying \eqref{multLagranage} for its minimal control with $\nu=1$ is called \emph{normal}. Notice that the case $\nu=0$ means that $u$ is a critical point of the end-point map. 

The covector $\lambda_1$ can be interpolated for times $t \in [0,1]$ yielding a lift of the curve $\gamma$ in the cotangent bundle. In particular in the normal case, the lift $\lambda : [0,1]\to T^*M$ solves a Hamiltonian differential equation. To state this fact more precisely, let us first define some objects: if $X$ is a vector field on $M$, we can associate to it a function $h_{X} : T^{*}M \to \R$ defined by $h_{X}(\lambda)=\langle \lambda,X \rangle$. In turn, if $h$ is a function on the cotangent bundle, its associated vector field $\vec{h}$ is defined by $\sigma(\cdot,\vec{h})=dh$, where $\sigma$ is the canonical symplectic form on the cotangent bundle, which can be expressed in coordinates as $\sigma = \sum_{i=1}^{n}dp_i \wedge dq_i$. Finally, for a sub-Riemannian structure in $M$ given by a generating family as above, we define the Hamiltonian $H: T^*M \to \R$ by 
\begin{equation}
H(\lambda)=\frac{1}{2} \sum_{i=1}^m h_{X_i}(\lambda)^2 = \frac{1}{2} \sum_{i=1}^m \langle \lambda,X_i \rangle ^2.
\end{equation}
The following result is an immediate consequence of the characterization of energy minimizers by the Lagrange multipliers rule, or can also be seen as a version of the Pontryagin maximum principle in this setting, cf.\ \cite{agrabook}.
\begin{theorem}	\label{thmextremals}
	Let $\gamma$ be a horizontal path minimizing the energy between $x$ et $y$, and let $u$ be its minimal control. Then there exists a Lipschitz path $\lambda : [0,1] \to T^* M$ lifting $\gamma$ --- that is for all $t \in [0,1]$, $\lambda(t) \in T_{\gamma(t)}^*M$ --- such that 
	\begin{equation}
	\dot{\lambda}(t)=\sum_{i=1}^m u_i(t) \vec{h}_{X_i}(\lambda(t)) \quad \text{a.e. } t \in [0,1].
	\end{equation}
	Moreover, one of those two conditions is verified:
	\begin{itemize}		
	\item[(N)] for all $t \in [0,1]$ and all $i=1,\dots,m$, it holds $u_i(t)=h_{X_i}(\lambda(t))$, that is $\lambda$ is solution of the differential equation $\dot{\lambda}(t)=\vec{H}(\lambda(t))$;
	\item[(A)] for all $t \in [0,1]$ and all $i =1,\dots,m$, it holds $h_{X_i}(\lambda(t))=0$ and $\lambda(t) \neq 0$.
	\end{itemize}
\end{theorem}
The conditions (N) and (A) correspond to the normal and abnormal cases of Lagrange multipliers rule, with $\lambda(1)$ corresponding to the multiplier $\lambda_1$ in \eqref{multLagranage}.

\begin{rmk}\label{affine}
From \eqref{multLagranage} the set of normal Lagrange multipliers of a path is an affine space over the linear space generated by its abnormal ones. The same property holds for the corresponding lifts.
\end{rmk}
The Hamiltonian characterization in the normal case allows us to define an exponential map $\exp_x : T_x^* M \to M$, where $\exp_x(\lambda)=\pi \circ e^{\vec{H}}(\lambda)$, where $t \mapsto e^{t\vec{H}}$ denotes the one-parameter group of diffeomorphisms on the cotangent bundle given by the Hamiltonian flow. In other words $\exp_x(\lambda)$ is the extremity at time $1$ of the normal geodesic whose lift verifies $\lambda(0)=\lambda$, that is parametrized by constant speed equal to $\sqrt{2H(\lambda)}$. 
Normal paths are locally length minimizing, and hence are geodesics. We assume that $(M,d_{SR})$ is complete, so that $\vec{H}$ is a complete vector field. 

Note that if a path is normal (resp.\ abnormal), any smaller segment is also normal (resp.\ abnormal) as the restriction of the lift verifies the same conditions.

The lift of a minimizing path given by Theorem \ref{thmextremals} is not necessarily unique and therefore the same horizontal path can be normal and abnormal at the same time. We will call a path \emph{strictly normal} if it is normal and it does not admit an abnormal lift, and \emph{strictly abnormal} if it is abnormal and it does not admit a normal lift. 

As a final remark, if a path is strictly normal, then its normal lift is unique. Indeed, if $\lambda$ and $\mu$ were two distinct normal lifts, then $\lambda(1)-\mu(1)$ would be an abnormal multiplier for this path (cf.\ Remark \ref{affine}).

\section{Branching geodesics}\label{s:branch}

A natural question is whether strictly normal paths can contain non-trivial abnormal subsegments. We will first show that this behaviour is linked to the occurrence of branching normal geodesics, and moreover that such an occurrence is actually equivalent to a jump in the rank of the differential of the end-point map. Then, in the next section, we will show a simple and natural example of this phenomenon. 

First let us define precisely what we will call branching here.

\begin{definition}
A normal geodesic $\gamma$ is \emph{branching} at time $t\in (0,1)$ if there exists a normal geodesic $\gamma'$ such that $\gamma|_{[0,t]}=\gamma'|_{[0,t]}$ and $\gamma|_{[0,t+\varepsilon]} \neq \gamma'|_{[0,t+\varepsilon]}$ for all $\varepsilon>0$.
\end{definition}

Let $\gamma$ be a normal geodesic. For $t \in [0,1]$ we define the set
\begin{equation}
\Pi_t = \{\lambda \in T_{\gamma(0)}^*M \mid \gamma(s) = \pi \circ e^{s\vec{H}}(\lambda) \quad \forall s \leq t \}.
\end{equation}
This set is a non-empty affine space corresponding to the initial normal covectors of the path $\gamma|_{[0,t]}$ given by Theorem \ref{thmextremals}. The set $\Pi_t$ is an isomorphic image of the set of normal Lagrange multipliers of  $\gamma|_{[0,t]}$ and, from Remark \ref{affine}, its dimension is the \emph{corank} of the path $\gamma|_{[0,t]}$, defined as the corank of the application $D_{u_t}E_x$, where $u_t$ is the minimal control of $\gamma|_{[0,t]}$. The function $t \mapsto \Pi_t$ for $t \in [0,1]$ is nonincreasing for the inclusion order and it is thus piecewise constant with some possible jumps where its dimension decreases.

\begin{definition}
The \emph{corank function} of $\gamma$ is the function that associates to a time $t \in [0,1]$ the corank of $\gamma|_{[0,t]}$, that is the function $t \mapsto \dim \Pi_t$. We say that $\gamma$ is \emph{rank-jumping} (or \emph{corank-jumping}) at time $t$ if there is a discontinuity in the corank function for this time.
\end{definition}
The corank function is nonincreasing and piecewise constant and moreover, from the lower semicontinuity of the rank and the $C^1$ regularity of the end-point map, it is left-continuous. We say that $\gamma$ is \emph{rank-jumping} (or \emph{corank-jumping}) at time $t$ if there is a discontinuity in the corank function for this time.

We can now state our theorem linking the phenomenon of branching normal geodesics with rank jumps, which at this point is very elementary to prove.

\begin{theorem}
A normal geodesic $\gamma$ branches at time $t \in (0,1)$ if and only if it is rank-jumping at time $t$.
\end{theorem}

\begin{proof}
Assume $\gamma$ branches at time $t$. Then the branching geodesic $\gamma'$ has an initial covector $\lambda'$ such that $\lambda' \in \Pi_t$ but $\lambda'\notin \Pi_{t+\varepsilon}$ for all $\varepsilon >0$, which means the rank jumps at $t$.
Conversely, if the rank jumps at time $t$, there is a covector $\lambda'$ in $\Pi_t$ that is not contained in $\Pi_s$ for $s>t$, and the path $\gamma'$ defined by $\gamma'(t)=\pi \circ e^{t\vec{H}}(\lambda')$ branches with $\gamma$ at time $t$.
\end{proof}

An immediate consequence is that a normal path can only branch a finite amount of times (up to the maximal corank of a path, which is the corank of the distribution), and furthermore $\gamma|_{[0,t]}$ must be abnormal.

If $\gamma$ is strictly normal, its corank is $0$ and we have the following corollary, corresponding to the situation encountered in the example from next section.

\begin{cor}
	\label{branchingtheorem}
	A strictly normal geodesic $\gamma$ is branching for some time $t \in (0,1)$ if and only if it contains a non-trivial abnormal subsegment that starts at time $0$. In particular if $t$ is the last branching time, $\gamma|_{[0,t]}$ is a maximal abnormal subsegment.
\end{cor}

In this situation, if $\gamma$ branches at time $t\in(0,1)$, then it branches in a whole family of distinct normal paths, parametrized by the abnormal Lagrange multipliers of the abnormal subsegment. To be more precise, let $A \subset T_{\gamma(t)}^*M$ be the set of abnormal Lagrange multipliers associated with the maximal abnormal subsegment $\gamma|_{[0,t]}$. Notice that $A \cup \{0\}$ is a vector space, and its dimension is the corank of the abnormal path $\gamma|_{[0,t]}$. Let $\lambda$ be the unique normal lift of $\gamma$. Then for all $\alpha\in A \cup \{0\}$ the family of curves
\begin{equation}\label{eq:family}
s\mapsto \gamma_\alpha(s) =\pi\circ e^{(s-t) \vec{H}}(\lambda(t)+\alpha), \qquad s \in [0,1]
\end{equation}
is a smooth family of normal geodesics, all coinciding with $\gamma$ on the subinterval $[0,t]$, and branching from it at time $t$.

We know that each of the paths in this family is locally minimizing since they are normal. Moreover, if we take a compact subfamily of those, the time at which they are minimizing, starting from the branching point, can be chosen uniformly. 
\begin{theorem}\label{minimalityfamily}
Let $\gamma_\alpha$ be a family of normal paths branching from $\gamma$ at time $t\in (0,1)$, as in \eqref{eq:family}. Then for any compact subset $A_0 \subset A \cup \{0\}$ there exists $\varepsilon>0$ such that $\gamma_a|_{[t-\varepsilon,t+\varepsilon]}$ is the unique length-minimizing path between its extremities, up to reparametrizations, for all $\alpha \in A_0$.
\end{theorem}
The proof of Theorem \ref{minimalityfamily} is a small adaptation of the proof for a single normal path, and it is an immediate consequence of the following more general result.
\begin{proposition}
	\label{minimalitythm}
	Let $x\in M$, and let $\Lambda \subset T_{x}^*M$ be a compact set. For $\lambda \in \Lambda$, consider the normal paths $\gamma^{\lambda}(t)=\pi \circ e^{t\vec{H}}(\lambda)$. Then there exists $\varepsilon >0$ such that, for all $\lambda \in \Lambda$, the restriction $\gamma^{\lambda}|_{[-\varepsilon,\varepsilon]}$ is the unique length-minimizing path between its extremities, up to reparametrizations.
\end{proposition}
\begin{proof}
We want to apply the following obvious extension of \cite[Corollary 4.64]{abb19}.
\begin{lemma}
	\label{minimalitylemma}
	Let $T>0$ and $a \in C ^{\infty}(M)$. Let $\Omega_0$ be an open subset of $M$ such that, for all $t \in [-T,T]$, the map $\pi \circ e^{t \vec{H}} \circ da|_{\Omega_0}$ is a diffeomorphism from $\Omega_0$ on its image $\Omega_t$. Let $\lambda_0 \in \mathcal{L}_0 \cap \pi^{-1}(\Omega_0)$ where $\mathcal{L}_0=\{d_z a \mid z\in M \}$ and define $\bar{\gamma}(t) = \pi \circ e^{t \vec{H}}(\lambda_0)$, for $t \in [-T,T]$. Then $\bar{\gamma}$ is the unique length-minimizing path, up to reparametrization, among all horizontal paths $\gamma : [-T,T] \to M$ with the same extremities and such that $\gamma(t) \in \Omega_t$ for all $t \in [-T,T]$.
\end{lemma}
To do it, for $\lambda \in \Lambda$, we construct a family of functions $a^{\lambda}$, continuous with respect to $\lambda$ such that $d_{x}a^{\lambda}=\lambda$. Indeed, the theorem being a local result, we can suppose to be in a coordinates system $(x_1,\dots,x_n)$ on a neighborhood of $x$, and if $\lambda=\sum_{i=1}^{n} \lambda_i dx_i$, we define $a^{\lambda}(x_1,\dots,x_n)=\sum_{i=1}^{n} \lambda_i x_i$.
	
Let $\Omega_0$ be a relatively compact neighborhood of $x$ which, for small $T$, contains $\gamma^{\lambda}|_{[-T,T]}$ for all $\lambda \in \Lambda$. Consider the maps $\phi_{t}^{\lambda}=\pi \circ e^{t \vec{H}} \circ da^\lambda|_{\Omega_0}$, for $t \in [-T,T]$ and $\lambda \in T_{x}^*M$, noting that they are continuous in $t$ and $\lambda$. For all $\lambda \in \Lambda$, we have $\phi_0^{\lambda}=\mathrm{Id}|_{\Omega_0}$, so by semi-continuity of the rank, and by the fact that $\overline{\Omega}_0$ is compact, there exists a neighborhood of $\{0\} \times {\Lambda}$ where $d_x\phi_{t}^{\lambda}$ in an isomorphism. By compactness of $\Lambda$, this neighborhood contains a set of the form $[-t_0,t_0] \times \Lambda$. By the inverse function theorem, and up to reducing $\Omega_0$, $\phi_{t}^{\lambda}$ is a diffeomorphism on its image for all $(t,\lambda) \in [-t_0,t_0] \times \Lambda$. Indeed, by compactness, the neighborhood of $x$ given by the inverse function theorem can be uniformly chosen for $(t,\lambda) \in [-t_0,t_0] \times \Lambda$, by using a quantitative version of the latter, see \cite[Theorem B.1.4] {rifford14}. 
	
	Let $K_1 \subset \Omega_0$ be a compact neighborhood of $x$. By continuity, there exists a neighborhood of $\{0\} \times \Lambda$ such that $K_1 \subset \Omega_t^\lambda=\phi_{t}^{\lambda}(\Omega_0)$ for all $(t,\lambda)$ in this neighborhood. Since $\Lambda$ is compact, we get $t_1 \in (0,t_0]$ such that $K_1 \subset \Omega_t^{\lambda}$ for all $t\in [-t_1,t_1]$ and all $\lambda \in \Lambda$.  Let then $K_2$ be a compact neighborhood of $x$ included in the interior of $K_1$ and we find $t_2 \in (0,t_1]$ such that $\gamma^{\lambda}(t) \in {K_2}$ for all $\lambda \in \Lambda$ and all $t \in [-t_2,t_2]$.	
	Finally, let us pose $\delta = d_{SR}(K_2,M \setminus K_1) >0$, and 
\begin{equation}
	\varepsilon = \min \left(t_2,\frac{\delta}{4 \sqrt{2\max_{\lambda \in \Lambda}H(\lambda)}}\right).
	\end{equation}	
	Let $\lambda \in \Lambda$ and $\gamma$ be a horizontal path defined for $[-\varepsilon,\varepsilon]$ such that $\gamma(-\varepsilon)=\gamma^{\lambda}(-\varepsilon)$ and $\gamma(\varepsilon)=\gamma^{\lambda}(\varepsilon)$, but whose image $\Gamma$ is distinct from the image of $\gamma^{\lambda}$. If $\Gamma \subset K_1$, then $\gamma(t) \in \Omega_t^{\lambda}$ for all $t$, and we can thus apply Lemma \ref{minimalitylemma} to conclude that $\ell(\gamma)>\ell(\gamma^{\lambda}|_{[-\varepsilon,\varepsilon]})$. Otherwise, there exists $t^*\in [-\varepsilon,\varepsilon]$ such that $\gamma(t^*) \notin K_1$. Then, since $\gamma(-\varepsilon) \in K_2$, we have:
	\begin{equation}
	l(\gamma) \geq \delta  > 2 \varepsilon \sqrt{2H(\lambda)} = \ell(\gamma^{\lambda}|_{[-\varepsilon,\varepsilon]}).\qedhere
	\end{equation}
\end{proof}

\section{An example of branching strictly normal geodesic}\label{s:example1}

Let us stress that normal geodesics cannot branch in real-analytic sub-Riemannian structures, that is when the corresponding Hamiltonian function is real-analytic. In fact in this case, by the Cauchy-Kowalevski theorem, normal geodesics, which are projections of the solutions of the Hamiltonian equation, are real-analytic paths. By the principle of permanence, two distinct real-analytic paths cannot be equal on a segment. That is, the following well-known fact holds:
\begin{proposition}
	If $H$ is real-analytic, normal geodesic cannot branch.
\end{proposition}

For building an example, we need to find a smooth, but non real-analytic structure, in which there is an abnormal geodesics that becomes strictly normal. A natural idea is to start from a structure admitting non-trivial abnormal geodesics (the simplest example being the flat Martinet structure) and ``glue'' it to a structure that do not admit non-trivial abnormal paths, like the Heisenberg structure. In fact, this works exactly as stated, and this is the idea that led us to the discovery of branching geodesics.

Let $\theta : \R \to [0,1]$ be a smooth non-decreasing function such that $\theta(t)=0$ if $t\leq 0$, $\theta(t)>0$ if $t>0$, and $\theta(t)=1$ if $t \geq 1$. Let $A(x,y)=x\theta(y) + x^2 \theta(1-y)$. Consider a rank $2$ sub-Riemannian structure on $\R^3$ defined by the following vector fields:
\begin{equation}
X= \partial_x \qquad Y=\partial_y + A(x,y) \partial_z,
\end{equation}
so that we have a flat Martinet structure on the half-space $y \leq 0$ and a Heisenberg one for $y \geq 1$. The Lie bracket between those vector fields is:
\begin{equation}
[X,Y]=B(x,y)\partial_z, \qquad\text{where}\qquad B(x,y) = \partial_x A(x,y) = [\theta(y) + 2x  \theta(1-y)],
\end{equation}
so that $\Span \{X,Y,[X,Y]\}= \mathbb{R}^3$, except on the so-called \emph{Martinet surface}
\begin{equation}
\Sigma = \{(x,y,z) \in \mathbb{R}^3 \mid B(x,y)=0\}.
\end{equation}
At all points in $\Sigma$, we have $[X,[X,Y]] = 2 \theta (1-y) \partial_z \neq 0$, therefore $\Sigma$ is smooth and the distribution is bracket-generating. It is well-known that abnormal paths for this distribution are exactly the horizontal ones contained in $\Sigma$, see for example \cite{montgomery02,rifford14}.
To characterize normal geodesics, the Hamiltonian function is
\begin{equation}
H=\frac12 \left(h_{X_1}^2+h_{X_2}^2\right)=\frac12 \left(p_x^2+(p_y + A(x,y)p_z)^2\right).
\end{equation}
The Hamiltonian vector field is thus, in coordinates $(x,y,z,p_x,p_y,p_z)$:
\begin{equation}
\label{hamiltonien}
\vec{H}= \begin{pmatrix} \frac{\partial H}{\partial p_x}\\[0,3em]
 \frac{\partial H}{\partial p_y} \\[0,3em]
 \frac{\partial H}{\partial p_z} \\[0,3em]
 -\frac{\partial H}{\partial x} \\[0,3em]
 -\frac{\partial H}{\partial y} \\[0,3em]
 -\frac{\partial H}{\partial z} \\[0,3em]
 \end{pmatrix}
 =
 \begin{pmatrix} p_x\\[0,3em]
 	p_y + A(x,y)p_z  \\[0,3em]
 	(p_y + A(x,y)p_z )A(x,y) \\[0,3em]
 	-p_z (p_y + A(x,y)p_z )B(x,y) \\[0,3em]
 	-p_z(p_y + A(x,y)p_z )\partial_y A(x,y) \\[0,3em]
 	0 \\[0,3em]
 \end{pmatrix}.
 \end{equation}
In particular, the path in the cotangent bundle $(0,t,0,0,1,0)$, for $t \in \R$ is an integral curve of $\vec{H}$, and therefore the lift of the normal geodesic $\gamma(t) = (0,t,0)$. For $t<0$, its projection $\gamma$ is contained in the Martinet surface $\Sigma$, and therefore this part of the curve is abnormal. Indeed, for every $\alpha \neq 0$, the path $(0,t,0,0,0,\alpha)$ is an abnormal lift of this geodesic. As soon as $t>0$, however, $\gamma$ is not contained in $\Sigma$ and therefore any such a segment is strictly normal. It is quite natural for this to happen as the Heisenberg structure has no abnormal geodesics. So if we consider the path $\gamma(t)=(0,t,0)$ for $t \in  [-T,T]$ for some $T>0$, it has a maximal abnormal subsegment $[-T,0]$ and therefore by Corollary \ref{branchingtheorem}, it branches at time $t=0$.

What happens is that, starting at $t=-T$ and until $t=0$, the differential of the end-point map has a 1-dimensional cokernel for the corresponding control, which is the family of covectors $(0,0,\alpha)$ for $\alpha \in \R$, and thus the space of initial covectors of normal lifts for this path is the 1-dimensional affine space $\{(0,1,\alpha)\mid \alpha \in \R\}$. Once time $t=0$ is attained the abnormal geodesic can still be prolonged (as the trace of the distribution in $\Sigma$) but it loses its normal status, becoming strictly abnormal. Meanwhile, the 1-dimensional family of normal lifts can be prolonged yielding a family of distinct geodesics, which are all strictly normal. In Figure \ref{branchement} we computed, using the Euler method, some of those geodesics $\gamma_\alpha$, for different values of $\alpha$.


\begin{figure}[ht]
	\centering
\includegraphics[width=\textwidth]{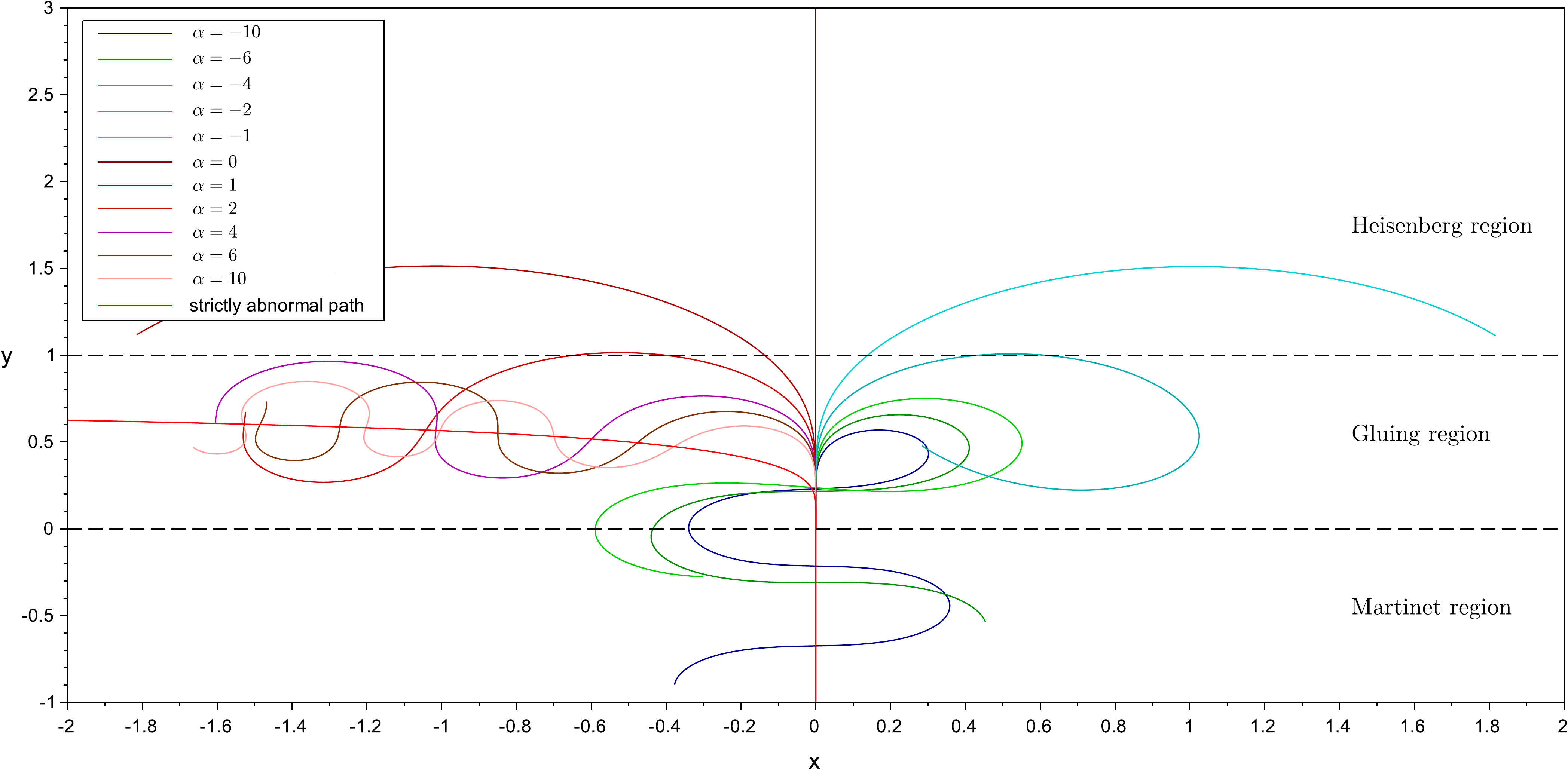}
	\caption{Numerical plot of the branching geodesics $\gamma_\alpha$, projected on the $xy$ plane. Notice that the abnormal path lies in the Martinet surface, which must bend in order to avoid the Heisenberg region.}\label{branchement}
\end{figure}

Finally, we observe that the collection of those normal geodesics do describe an embedded surface of $\R^3$, at least locally around the $y$-axis (which is the normal geodesic with initial covector $(0,1,0)$) as shown in this result:

\begin{proposition}\label{p:spray}
	The map $\Phi : \mathbb{R}^2 \to \mathbb{R}^3$ defined by $\Phi(t,\alpha)=\pi\circ e^{(t+T)\vec{H}}(\lambda_\alpha)$, where $\lambda_{\alpha}$ is the initial covector $(0,1,\alpha)$ at point $(0,-T,0)$, is an embedding on a neighborhood of any point $(t,0)$ with $t>0$.
\end{proposition}

\begin{proof}
Since $\Phi(t,0) = (0,t,0)$, we have $\frac{\partial \Phi}{\partial t} (t,0) = \partial_y$.
		So we just have to show that $\frac{\partial \Phi}{\partial \alpha} (t,0)$ is independent from $\partial_y$, which we will do by pointing out that its $x$ component is non zero. Let 
		\begin{equation}
t\mapsto		(x^\alpha(t),y^\alpha(t),z^\alpha(t),p_x^\alpha(t),p_y^\alpha(t),p_z^\alpha(t))
		\end{equation}
		be the solution of Hamilton's equation \eqref{hamiltonien}, with initial condition $(0,-T,0,0,1,\alpha)$ at time $t=0$, in such a way that $\Phi(t,\alpha) = (x^\alpha(t),y^\alpha(t),z^\alpha(t))$. We observe that $p_z^{\alpha}(t) = \alpha$ for all times, and thus:
\begin{equation}
x^{\alpha}(t)=\int_0^t p_x^{\alpha}(s)ds=\int_0^t \left(\int_0^s \dot{p}_x^{\alpha}(\tau) d\tau\right)ds =-\alpha \int_0^t \left(\int_0^s f(\alpha,\tau)d\tau\right)ds.
\end{equation}
Here $f(\alpha,\tau)$ is a smooth function whose expression can be obtained from \eqref{hamiltonien}. We notice only that $f(0,\tau)=\theta(\tau)$. So
	\begin{equation}
\left.	\frac{\partial x^{\alpha}(t)}{\partial \alpha} \right|_{\alpha=0} = - \int_0^t \int_0^s\theta(\tau)d\tau ds,
	\end{equation}
	which is non-zero by the properties of $\theta$.
\end{proof}

\begin{remark}
As anticipated in Section \ref{s:magnetic}, any smooth $A$ such that $A(0,y)=0$, and such that $B=\partial_x A$ satisfies $B(x,y) =x$ for $y<0$, and $B(0,y)>0$ when $y>0$, yields a one-parameter family of branching geodesics, verifying Proposition \ref{p:spray}.
\end{remark}

\section{Normal geodesics with multiple branching}\label{s:multiple}

In the example from the previous section, the corank function of a path $\gamma_{\alpha}$ starting at $t=-T$ is equal to one for $t\leq 0$ and zero for $t>0$. From this, we can easily construct any kind of corank function, and therefore any kind of normal branching. To do that, we consider  a suitable product of our example. 

The product $M_1 \times M_2$ of sub-Riemannian manifolds $M_1$ and $M_2$ has a product sub-Riemannian structure simply defined as the direct sum of the distributions and the metrics in $M_1$ and $M_2$. It is then easy to see that a path $\gamma$ is a geodesic in $M$ if and only if $\gamma_1$ and $\gamma_2$ are geodesics in $M_1$ and $M_2$, where $\gamma_i$ is the projection of $\gamma$ on $M_i$. Moreover, if $\lambda_1$ and $\lambda_2$ are normal Lagrange multipliers for $\gamma_1$ and $\gamma_2$ respectively, $(\lambda_1,\lambda_2)$ is a normal Lagrange multiplier for $\gamma$, and the converse is true. The same is true for abnormal multipliers, so the corank of a path $\gamma$ is the sum of the coranks of the $\gamma_i$. From these observations, we get the following result, which gives a complete description of what kind of branching one can expect:
\begin{theorem}
	Let $f:[0,1] \to \N$ be a nonincreasing left-continuous function. Then there exists a sub-Riemannian manifold $M$ and a normal geodesic $\gamma$ on $M$ such that its corank function coincides with $f$. In particular $\gamma$ branches at the jumps of $f$.
\end{theorem}
\begin{proof}
	Constant paths have corank equal to the corank of the distribution, therefore the corank of $M$ must be equal to $f(0)$. So we define $M$ as the sub-Riemannian product $M=N^{f(0)}$, where $N$ is the sub-Riemannian structure in $\R^3$ defined in Section \ref{s:example1}. Denote then by $a=f(0)-f(0^+)$ ($f(0^+)$ being the right limit of $f$ at $0$), $b=f(1)$ and $t_1,\dots,t_k$ the times of the discontinuities of $f$ in $(0,1)$, each one repeated multiple times according to the amplitude of the discontinuity. Denote as previously by $\gamma$ the path in $N$ defined by $\gamma(t)=(0,t,0)$ for $t\in \R$ and take $\gamma_0$ any strictly normal geodesic with no non-trivial abnormal subsegments in $N$ (for example $\gamma_0=\gamma|_{[0,1]}$).
	
	If we define the path $\tilde{\gamma} : [0,1] \to M$ by:
	\begin{equation}
	\tilde{\gamma}(t)=(\underbrace{\gamma_0(t),\dots,\gamma_0(t)}_{a \text{ times}},\gamma(t-t_1),\gamma(t-t_2),\dots,\gamma(t-t_k),\underbrace{0_{\R^3},\dots,0_{\R^3}}_{b \text{ times}}),
	\end{equation}
then the corank function of $\tilde{\gamma}$ is exactly $f$.

	As a remark, we also have a full description of the normal geodesics branching from $\tilde{\gamma}$, given, for $\alpha_1,\dots,\alpha_k \in \R$, by:
	$$\tilde{\gamma}_{(\alpha_1,\dots,\alpha_k)}(t)=(\gamma_0(t),\dots,\gamma_0(t),\gamma_{\alpha_1}(t-t_1),\gamma_{\alpha_2}(t-t_2),\dots,\gamma_{\alpha_k}(t-t_k),0_{\R^3},\dots,0_{\R^3}),$$
	where $\gamma_{\alpha}$ are the already defined branching geodesics in $N$, from Section \ref{s:example1}.
\end{proof}

\subsection*{Acknowledgments}
This work was supported by the Grants ANR-15-CE40-0018 and ANR-15-IDEX-02. While this note was in preparation, N. Juillet shared with us an example similar to the one of Section \ref{s:example1}, in a rank-varying structure on $\mathbb{R}^4$.  We thank the anonymous referee for carefully reading the paper, and suggesting the simple description of branching in terms of magnetic field presented in Section \ref{s:magnetic}.

\bibliography{bibliographie}

\begin{bibdiv}
\begin{biblist}

\bib{abb19}{book}{
      author={Agrachev, Andrei~A.},
      author={Barilari, Davide},
      author={Boscain, Ugo},
       title={A comprehensive introduction to sub-{R}iemannian geometry},
      series={Cambridge Studies in Advanced Mathematics},
   publisher={Cambridge University Press},
        date={2019},
}

\bib{agrabook}{book}{
      author={Agrachev, Andrei~A.},
      author={Sachkov, Yuri~L.},
       title={Control theory from the geometric viewpoint},
      series={Encyclopaedia of Mathematical Sciences},
   publisher={Springer-Verlag, Berlin},
        date={2004},
      volume={87},
        ISBN={3-540-21019-9},
         url={https://doi.org/10.1007/978-3-662-06404-7},
        note={Control Theory and Optimization, II},
}

\bib{BG-JEMS}{article}{
      author={Baudoin, Fabrice},
      author={Garofalo, Nicola},
       title={Curvature-dimension inequalities and {R}icci lower bounds for
  sub-{R}iemannian manifolds with transverse symmetries},
        date={2017},
        ISSN={1435-9855},
     journal={J. Eur. Math. Soc. (JEMS)},
      volume={19},
      number={1},
       pages={151\ndash 219},
         url={https://doi.org/10.4171/JEMS/663},
      review={\MR{3584561}},
}

\bib{BKS1}{article}{
      author={Balogh, Zolt\'{a}n~M.},
      author={Krist\'{a}ly, Alexandru},
      author={Sipos, Kinga},
       title={Geometric inequalities on {H}eisenberg groups},
        date={2018},
        ISSN={0944-2669},
     journal={Calc. Var. Partial Differential Equations},
      volume={57},
      number={2},
       pages={Art. 61, 41},
         url={https://doi.org/10.1007/s00526-018-1320-3},
      review={\MR{3774461}},
}

\bib{BKS2}{article}{
      author={Balogh, Zolt\'{a}n~M.},
      author={Krist\'{a}ly, Alexandru},
      author={Sipos, Kinga},
       title={Jacobian determinant inequality on corank 1 {C}arnot groups with
  applications},
        date={2019},
        ISSN={0022-1236},
     journal={J. Funct. Anal.},
      volume={277},
      number={12},
       pages={108293, 36},
         url={https://doi.org/10.1016/j.jfa.2019.108293},
      review={\MR{4019096}},
}

\bib{BR-interpolation}{article}{
      author={Barilari, Davide},
      author={Rizzi, Luca},
       title={Sub-{R}iemannian interpolation inequalities},
        date={2019},
        ISSN={0020-9910},
     journal={Invent. Math.},
      volume={215},
      number={3},
       pages={977\ndash 1038},
         url={https://doi.org/10.1007/s00222-018-0840-y},
      review={\MR{3935035}},
}

\bib{corner}{article}{
      author={Hakavuori, Eero},
      author={Le~Donne, Enrico},
       title={Non-minimality of corners in subriemannian geometry},
        date={2016},
        ISSN={0020-9910},
     journal={Invent. Math.},
      volume={206},
      number={3},
       pages={693\ndash 704},
         url={https://doi.org/10.1007/s00222-016-0661-9},
      review={\MR{3573971}},
}

\bib{LS-memoir}{article}{
      author={Liu, Wensheng},
      author={Sussman, H\'{e}ctor~J.},
       title={Shortest paths for sub-{R}iemannian metrics on rank-two
  distributions},
        date={1995},
        ISSN={0065-9266},
     journal={Mem. Amer. Math. Soc.},
      volume={118},
      number={564},
       pages={x+104},
         url={https://doi.org/10.1090/memo/0564},
      review={\MR{1303093}},
}

\bib{Milman-QCD}{misc}{
      author={Milman, Emanuel},
       title={The quasi curvature-dimension condition with applications to
  sub-{R}iemannian manifolds},
        date={2019},
}

\bib{montgomery02}{book}{
      author={Montgomery, Richard},
       title={A tour of subriemannian geometries, their geodesics and
  applications},
      series={Mathematical Surveys and Monographs},
   publisher={American Mathematical Society},
     address={Providence, RI},
        date={2002},
      volume={91},
        ISBN={0-8218-1391-9},
}

\bib{Mont-iso}{article}{
      author={Montgomery, R.},
       title={Isoholonomic problems and some applications},
        date={1990},
        ISSN={0010-3616},
     journal={Comm. Math. Phys.},
      volume={128},
      number={3},
       pages={565\ndash 592},
         url={http://projecteuclid.org/euclid.cmp/1104180539},
      review={\MR{1045885}},
}

\bib{Mont-abnormal}{article}{
      author={Montgomery, Richard},
       title={Abnormal minimizers},
        date={1994},
        ISSN={0363-0129},
     journal={SIAM J. Control Optim.},
      volume={32},
      number={6},
       pages={1605\ndash 1620},
         url={https://doi.org/10.1137/S0363012993244945},
      review={\MR{1297101}},
}

\bib{rifford14}{book}{
      author={Rifford, Ludovic},
       title={Sub-{R}iemannian geometry and optimal transport},
      series={Springer Briefs in Mathematics},
   publisher={Springer, Cham},
        date={2014},
        ISBN={978-3-319-04803-1; 978-3-319-04804-8},
         url={http://dx.doi.org/10.1007/978-3-319-04804-8},
}

\bib{R-Bourbaki}{incollection}{
      author={Rifford, Ludovic},
       title={Singuli\`eres minimisantes en g\'{e}om\'{e}trie
  sous-{R}iemannienne},
        date={2017},
       pages={Exp. No. 1113, 277\ndash 301},
        note={S\'{e}minaire Bourbaki. Vol. 2015/2016. Expos\'{e}s 1104--1119},
      review={\MR{3666029}},
}

\bib{V-Bourbaki}{incollection}{
      author={Villani, C\'{e}dric},
       title={In\'{e}galit\'{e}s isop\'{e}rim\'{e}triques dans les espaces
  m\'{e}triques mesur\'{e}s [d'apr\`es {F}. {C}avalletti \& {A}. {M}ondino]},
        date={2019},
       pages={Exp. No. 1127, 213\ndash 265},
         url={https://doi.org/10.24033/ast},
        note={S\'{e}minaire Bourbaki. Vol. 2016/2017. Expos\'{e}s 1120--1135},
      review={\MR{3939278}},
}

\end{biblist}
\end{bibdiv}

\end{document}